\documentclass[12pt]{amsart}
\usepackage[utf8]{inputenc}
\usepackage{graphicx}
\usepackage{epstopdf}
\usepackage{inputenc}
\usepackage{amsmath}
\usepackage{tikz}
\usepackage{xcolor}
\usepackage{amssymb}
\newtheorem{theorem}{Theorem}
\newtheorem*{theorem*}{Theorem}

\newtheorem{proposition}{Proposition}

\newtheorem{lemma}{Lemma}
\newtheorem{conjecture}{Conjecture}
\newtheorem*{acknowledgements*}{Acknowledgements}

\makeatletter
\def\blfootnote{\gdef\@thefnmark{}\@footnotetext}
\makeatother

\def\house#1{\setbox1=\hbox{$\,#1\,$}%
\dimen1=\ht1 \advance\dimen1 by 2pt \dimen2=\dp1 \advance\dimen2 by 2pt
\setbox1=\hbox{\vrule height\dimen1 depth\dimen2\box1\vrule}%
\setbox1=\vbox{\hrule\box1}%
\advance\dimen1 by .4pt \ht1=\dimen1
\advance\dimen2 by .4pt \dp1=\dimen2 \box1\relax}

\begin{document}
\title[On the simultaneous $3$-divisibility of class numbers]{On the simultaneous $3$-divisibility of class numbers of triples of imaginary quadratic fields}
\author{JAITRA CHATTOPADHYAY and SUBRAMANI MUTHUKRISHNAN}
\address[Jaitra Chattopadhyay]{Harish-Chandra Research Institute, HBNI, 
Chhatnag Road, Jhunsi, Allahabad - 211019, INDIA.}

  \address[Subramani Muthukrishnan]{Indian Institute of Information Technology Design and Manufacturing Kancheepuram, Vandalur-Kelambakkam Road, Chennai - 600127, India.}
\email[Jaitra Chattopadhyay]{chat.jaitra@gmail.com}
\email[Subramani Muthukrishnan]{subramaniat2010@gmail.com}

\begin{abstract}
Let $k \geq 1$ be a cube-free integer with $k \equiv 1 \pmod {9}$ and $\gcd(k, 7\cdot 571)=1$. 
We prove the existence of infinitely many triples of imaginary quadratic fields 
$\mathbb{Q}(\sqrt{d})$, $\mathbb{Q}(\sqrt{d+1})$ and $\mathbb{Q}(\sqrt{d+k^2})$ with $d \in \mathbb{Z}$ 
such that the class number of each of them is divisible by $3$. 
This affirmatively answers a weaker version of a conjecture of Iizuka \cite{iizuka-jnt}. 
%
%
\end{abstract}
\maketitle

\section{Introduction}

For a number field $K$, we denote the ideal class group of $K$ by $Cl_{K}$ and the class number by $h_{K}$. 

Ankeny and Chowla \cite{AC} and Nagell \cite{nagell}, among others, proved the existence of infinitely many imaginary 
quadratic fields with class numbers divisible by a given positive integer $n$. 
Later, Weinberger \cite{berger}, Yamamoto \cite{moto} and several other mathematicians proved the analogue for real quadratic fields.


The problem of simultaneous divisibility of class numbers of tuples of quadratic fields was considered by
Scholz \cite{scholz}, who proved that for a square-free integer $d \geq 2$, if $h_{\mathbb{Q}(\sqrt{d})}$ 
is divisible by $3$, then so is $h_{\mathbb{Q}(\sqrt{-3d})}$. 
In \cite{komatsu-acta}, Komatsu extended Scholz's $3$-divisibility result for $\mathbb{Q}(\sqrt{d})$ and 
$\mathbb{Q}(\sqrt{md})$ for a non-zero integer $m$ and in 
\cite{komatsu-ijnt}, he generalized the theorem for the divisibility of class numbers by a 
given integer $n \geq 2$ for pairs of imaginary quadratic fields.

Very recently, Iizuka \cite{iizuka-jnt} considered a slight variant of the problems considered in \cite{komatsu-acta} and \cite{komatsu-ijnt} 
(see also \cite{iizuka1} and \cite{ikn-preprint}). He proved that there exist infinitely many imaginary quadratic fields 
$\mathbb{Q}(\sqrt{d})$ and $\mathbb{Q}(\sqrt{d+1})$ with $d \in \mathbb{Z}$ and class numbers divisible by $3$. 
Moreover, in \cite{iizuka-jnt}, Iizuka made the following conjecture.

\begin{conjecture} \cite{iizuka-jnt}\label{conj}
Let $m \geq 1$ be an integer and let $\ell \geq 3$ be a prime number. Then there exist infinitely many tuples $\{\mathbb{Q}(\sqrt{d}), \mathbb{Q}(\sqrt{d + 1}), \ldots ,\mathbb{Q}(\sqrt{d + m})\}$ of quadratic fields with $d \in \mathbb{Z}$ such that $\ell$ divides the class numbers of all of them.
\end{conjecture}

In this paper, we extend Iizuka's result from pairs to certain triples of imaginary quadratic fields. 
This addresses a weaker version of Conjecture \ref{conj} for $\ell = 3$. To the best of our knowledge, 
this is the first result in the direction of the simultaneous divisibility of class numbers of quadratic fields, 
taken three at a time. The precise statement of our main theorem is as follows.

\begin{theorem}\label{main-theorem}
Let $k \geq 1$ be a cube-free integer such that $k \equiv 1 \pmod 9$ and $\gcd(k, 7\cdot 571)=1$. Then there exist infinitely many triples of imaginary quadratic fields $\mathbb{Q}({\sqrt{d}})$, $\mathbb{Q}({\sqrt{d+1}})$ and $\mathbb{Q}({\sqrt{d+k^2}})$ with $d \in \mathbb{Z}$ such that $3$ divides each of $h_{\mathbb{Q}(\sqrt{d})}$, $h_{\mathbb{Q}(\sqrt{d + 1})}$ and $h_{\mathbb{Q}(\sqrt{d + k^{2}})}$.
\end{theorem}

\section{Preliminaries}

For any number field $K$ and prime $p$, the $p$-rank of $K$, denoted by ${\rm{rk}}_{p}(Cl_{K})$, is the dimension of 
$Cl_{K}/pCl_{K}$ as a $\mathbb{Z}/p\mathbb{Z}$ vector space.

For $p=3$, Scholz \cite{scholz} proved the following ``reflection" principle for the $3$-ranks of the class groups of quadratic fields.

%
%

\begin{theorem} \cite{scholz}\label{scholz-reflection}
Let $d > 0$ be a square-free integer. If $r$ and $s$ are the $3$-ranks of the class groups of $\mathbb{Q}(\sqrt{d})$ and 
$\mathbb{Q}(\sqrt{-3d})$, respectively, then $r \leq s \leq r+1$. In particular, if $3$ divides $h_{\mathbb{Q}(\sqrt{d})}$, 
then $3$ also divides $h_{\mathbb{Q}(\sqrt{-3d})}$.
\end{theorem}

%
%
%
%
%
%

\medskip

The proof of Theorem \ref{main-theorem} is based on constructing unramified cyclic cubic extensions of certain quadratic fields. 
Suppose $f(X) \in \mathbb{Z}[X]$ is an irreducible polynomial of degree $3$ whose discriminant $D(f)$ is not a perfect square
and let $E$ be the splitting field of $f(x)$ over $\mathbb{Q}$.
Then $E$ is Galois over both $\mathbb{Q}$ and the quadratic field $F=\mathbb{Q}(\sqrt{D(f)})$.  
Since $E$ is a cyclic cubic extension of $F$, if also $E$ is unramified over $F$ then, by class field theory,
$3$ divides the class number of $F$.

Since $E$ is a Galois extension of $F$ of odd degree, $E$ is unramified at the infinite primes of $F$.  For the 
finite primes, we have the following 

\medskip

\begin{lemma} \cite{kishi-miyake}\label{proved-n-kishi-miyake}
Let $f(X)\in \mathbb{Z}[X]$ be a cubic irreducible polynomial and let $E$ be the splitting field of $f$ over $\mathbb{Q}$. 
Assume that $D(f)$ is not a perfect square and let $F=\mathbb{Q}(\sqrt{D(f)})$. For a prime number $p$, let $\wp_{F}$ be a 
prime ideal in $\mathcal{O}_{F}$ lying above $p$. Let $\alpha$ be a root of $f$ and let $K = \mathbb{Q}(\alpha)$. 
Then $\wp_{F}$ is ramified in $E$ if and only if $p$ is totally ramified in $K$.
\end{lemma}

This lemma, stated in \cite{kishi-miyake} without proof, follows easily since both conditions are equivalent to
3 dividing the ramification index of $p$ in $E$.

\medskip


%



By Lemma \ref{proved-n-kishi-miyake}, questions of which primes ramify from $F$ to $E$ are reduced to questions of which
primes $p$ ramify totally from $\mathbb{Q}$ to the cubic field defined by a root of $f(x)$.  The next lemma 
(as presented in \cite{kishi-miyake} as a consequence of Theorem $1$ in \cite{nart}) answers such
questions for certain polynomials $f(x)$.
As usual, for a prime $p$ the $p$-adic valuation of the integer $n$ is denoted $v_{p}(n)$.

\begin{lemma} \cite{kishi-miyake}\label{lemma-taken-from-kishi-miyake}
Let $f(X)=X^3 - aX - b \in \mathbb{Z}[X]$ be an irreducible polynomial over $\mathbb{Q}$ 
whose discriminant $D(f)$ is 
not a perfect square and such that either 
$v_{p}(a) < 2$ or $v_{p}(b) < 3$ for every prime number $p$.  Let $\alpha$ be a root of $f$ and let $K=\mathbb{Q}(\alpha)$. 

\begin{enumerate}
\item If $q \neq 3$ is a prime, then $q$ is totally ramified in $K$ if and only if $1 \leq v_{q}(b) \leq v_{q}(a)$.

\item The prime $3$ is totally ramified in $K$ if and only if one of the following conditions holds:

\begin{enumerate}
\item[(i)] $1 \leq v_{3}(a) \leq v_{3}(b)$,

\item[(ii)] $3 \mid a, a \not\equiv 3 \pmod {9}, 3 \nmid b \mbox{ and } b^{2} \not\equiv a+1 \pmod {9}$,

\item[(iii)] $a \equiv 3 \pmod {9}, 3 \nmid b \mbox{ and } b^{2} \not\equiv a+1 \pmod {27}$.
\end{enumerate}
\end{enumerate}
\end{lemma}

\medskip

Using Lemma \ref{proved-n-kishi-miyake} and Lemma \ref{lemma-taken-from-kishi-miyake}, we construct two families of quadratic fields with class numbers divisible by $3$ as follows.

\medskip

\begin{proposition}\label{1st-class-number-divisibility}
For any non-zero integer $t$ with $t \not\equiv 0 \pmod {3}$, the class number of the quadratic field 
$\mathbb{Q}(\sqrt{3t(3888t^2 + 108t + 1)})$ is divisible by $3$.
\end{proposition}

\begin{proof}
If $t \not\equiv 0 \pmod {3}$ then $3t(3888t^2 + 108t + 1)$ is divisible by 3 but not by 9, hence is not 
a perfect square, so $\mathbb{Q}(\sqrt{3t(3888t^2 + 108t + 1)})$ is indeed a quadratic field.  
The polynomial $f(X)=X^3 -3\cdot (108t + 1)X -2 \in \mathbb{Z}[X]$ is irreducible by the rational root theorem
and has discriminant $D(f)= 2^{2}\cdot 3^{5}\cdot t\cdot (3888t^2 + 108t + 1)$.

Let $\alpha$ be a root of $f$ and let $K=\mathbb{Q}(\alpha)$. 
For each prime number $p$, an easy check using Lemma \ref{lemma-taken-from-kishi-miyake} shows that $p$ is not totally ramified in $K$.
Hence, by Lemma \ref{proved-n-kishi-miyake}, we conclude that 
the splitting field $E$ of $f(x)$ over $\mathbb{Q}$ is an unramified cyclic cubic extension of $F = \mathbb{Q}(\sqrt{3t(3888t^2 + 108t + 1)})$, so
the class number of $F$ is divisible by 3. 
\end{proof}

\medskip

Following a similar line of argument, we provide another family of quadratic fields with class numbers divisible by $3$ in the next proposition.

\begin{proposition}\label{2nd-class-number-divisibility}
Let $t \geq 1$ be an integer. Then the class number of the imaginary quadratic field $\mathbb{Q}(\sqrt{1 - 2916t^3})$ is divisible by $3$.
\end{proposition}

\begin{proof}
Since $1 - 2916t^{3} < 0$ for $t \geq 1$, the field $\mathbb{Q}(\sqrt{1 - 2916t^{3}})$ is indeed a quadratic field.
The polynomial $f(X) = X^3 - 27tX - 1 \in \mathbb{Z}[X]$ is irreducible by the rational root theorem and has 
discriminant $D(f)=3^3 (2916t^3 - 1).$

Let $F = \mathbb{Q}(\sqrt{D(f)})=\mathbb{Q}(\sqrt{3\cdot (2916t^3 - 1)})$. Since $\gcd(3, 2916t^{3} - 1)=1$, 
the integer $3\cdot (2916t^{3} - 1)$ is not a perfect square, so $\mathbb{Q}(\sqrt{D(f)})$ is also a quadratic field. 

Let $\alpha$ be a root of $f(x)$ and let $K=\mathbb{Q}(\alpha)$. 
As before, for each prime number $p$, an easy check using Lemma \ref{lemma-taken-from-kishi-miyake} shows that 
$p$ is not totally ramified in $K$. By Lemma \ref{proved-n-kishi-miyake} 
the splitting field $E$ of $f(x) $ over $\mathbb{Q}$ is an unramified cyclic cubic extension of
$F = \mathbb{Q}(\sqrt{3\cdot (2916t^3 - 1)})$ so $3$ divides the class number of the real quadratic field $F$.
By Theorem \ref{scholz-reflection}, we conclude that the class number of the imaginary quadratic field 
$\mathbb{Q}(\sqrt{-3\cdot 3\cdot (2916t^{3} - 1)}) = \mathbb{Q}(\sqrt{1 - 2916t^{3}})$ is also divisible by $3$. 
\end{proof}

\section{Proof of Theorem \ref{main-theorem}}

Let $k \geq 1$ be any fixed cube-free integer such that $k \equiv 1 \pmod 9$ and $\gcd(k, 7\cdot 571)=1$ and 
for an integer $t \geq 1$, consider the polynomial $f_{t}(X)=X^3 - 27tX - k \in \mathbb{Z}[X]$. 

If $f_{t}$ is reducible over $\mathbb{Q}$ for some $t$, then it has an integer root 
dividing $k$, so there are at most finitely many values of $t$ for which $f_{t}$ is reducible over $\mathbb{Q}$.
Also, the discriminant $D(f_{t})=27\cdot (2916t^{3} - k^{2})$ is a polynomial in $t$ which, since 
$k \neq 0$, has distinct roots in $\overline{\mathbb{Q}}$. 
By Siegel's theorem on integral points (see \cite{silverman}, Chapter 9, Theorem 4.3), $D(f_{t})$ is 
a perfect square for finitely many integers $t$.  
It follows that there is a positive integer $T > k$ such that $f_{t}$ is irreducible and $D(f_{t})$ 
is not a perfect square for all $t \ge T$.

Let $$\mathcal{N}=\{n \in \mathbb{Z} : n \equiv 2 \ (\text{mod } {9}) ,  n \equiv 1 \ (\text{mod } {k}) \text{ and } n \ge T  \}.$$

\noindent
For $n \in \mathcal{N}$, we have 
\begin{equation}\label{tp}
27\cdot n\cdot (3888n^2 + 108n + 1) \equiv 3^{3}\cdot 7\cdot 571 \not\equiv 0 \pmod {k}.
\end{equation}

Now, for $n \in \mathcal{N}$, let 
$$
t_{n}=n\cdot (3888n^2 + 108n + 1)
$$ 
and consider the polynomial 
$$
f_{t_{n}}(X)=X^3 - 27 t_n X - k
$$ 
over $\mathbb{Q}$. Since $t_{n} > T$, 
it follows that $f_{t_{n}}$ is irreducible over $\mathbb{Q}$ and $D(f_{t_{n}})$ is not a perfect square. 
Now, using \eqref{tp}, Lemma \ref{proved-n-kishi-miyake} and a quick check of the conditions of 
Lemma \ref{lemma-taken-from-kishi-miyake} as before, we conclude the splitting field $E$ of $f_{t_{n}}$ over 
$\mathbb{Q}$ is an unramified cyclic cubic extension over $\mathbb{Q}(\sqrt{D(f_{t_{n}})})$. Therefore, $3$ divides 
the class number of $\mathbb{Q}(\sqrt{D(f_{t_{n}})})$.\smallskip

Note that 
$\mathbb{Q}(\sqrt{D(f_{t_{n}})})=\mathbb{Q}(\sqrt{27\cdot (2916t_{n}^3 - k^2)})=\mathbb{Q}(\sqrt{3\cdot (2916t_{n}^3 - k^2)})$ 
is a real quadratic field. Consequently, Theorem \ref{scholz-reflection} yields that $3$ divides the 
class number of the imaginary quadratic field $\mathbb{Q}(\sqrt{-3\cdot 3\cdot (2916t_{n}^3 - k^2)})=\mathbb{Q}(\sqrt{k^2 - 2916t_{n}^3})$.

Also, Proposition \ref{1st-class-number-divisibility} implies that $3$ divides the class number of the 
real quadratic field $\mathbb{Q}(\sqrt{3t_{n}})$. Again from Theorem \ref{scholz-reflection}, 
we obtain that $3$ divides the class number of the imaginary quadratic field 
$\mathbb{Q}(\sqrt{-3\cdot 3t_{n}})=\mathbb{Q}(\sqrt{-t_{n}}) = \mathbb{Q}(\sqrt{-2916t_{n}^3})$. 
Together with Proposition \ref{2nd-class-number-divisibility} this shows that
$3$ divides the class numbers of $\mathbb{Q}(\sqrt{-2916t_{n}^3})$, 
$\mathbb{Q}(\sqrt{-2916t_{n}^3 + 1})$ and $\mathbb{Q}(\sqrt{-2916t_{n}^3 + k^2})$.

The family $ \mathcal{Q} = \{\mathbb{Q}(\sqrt{-t_{n}}) : n \in \mathcal{N}\} $
is infinite since there are infinitely many primes $q$ in $\mathcal{N}$ by Dirichlet's theorem 
and $\mathbb{Q}(\sqrt{-t_{q}})$ is ramified at $q$, completing the proof. $\hfill\Box$

\medskip

{\bf Acknowledgements.} It is a pleasure to thank Dr. Iizuka for sending us the manuscript \cite{ikn-preprint} on request. We are grateful to Prof. R. Thangadurai for his support and encouragement throughout the project. We sincerely thank him for going through the manuscript several times and giving valuable comments. We are greatly thankful to Prof. K. Srinivas for his careful reading and valuable comments on this paper. We gratefully acknowledge the anonymous referee for his/her valuable remarks that hugely improved the readability of the paper.


\end{document}